\documentclass[preprint,12 pt]{elsarticle}
\usepackage[british]{babel}
\usepackage{mathrsfs}
\usepackage{amsmath}
\usepackage{amssymb}
\usepackage{amsthm}
\usepackage{enumerate}
\usepackage{epic}
\usepackage{color}
\usepackage{lineno}

\numberwithin{equation}{section}
\makeatletter
\renewcommand{\subsection}{\@startsection
{subsection}{2}{0mm}{\baselineskip}{-0.25cm}
{\normalfont\normalsize\bf}}
\makeatother

\newtheorem{theorem}{Theorem}[section]
\newtheorem{proposition}[theorem]{Proposition}
\newtheorem{lemma}[theorem]{Lemma}

\def\F{\mathbf F}

\def\cC{\mathcal C}

\def\cF{\mathcal F}

\def\cP{\mathcal P}
\def\cX{\mathcal X}

\def\det{{\rm det}}

\newcommand{\ha}{\textstyle\frac{1}{2}}

\def\xfq2{{\cX(\mathbb{F}_{q^2})}}
\def\F+xfq2{{\cF^+(\mathbb{F}_{q^2})}}
\def\xpfq2{{\cX^+(\mathbb{F}_{q^2})}}
\def\xmfq2{{\cX^-(\mathbb{F}_{q^2})}}
\journal{Finite Field and Applications}
\begin{document}
\begin{frontmatter}

\title{Strongly regular graphs from integral point sets in even dimensional
affine spaces over finite fields}
\author{G\'{a}bor Korchm\'{a}ros}
\ead{gabor.korchmaros@unibas.it}
\address{Dipartimento di Matematica, Informatica ed Economia - Universit\`{a} degli Studi della Basilicata - Viale dell'Ateneo Lucano 10 - 85100 Potenza (Italy),}
\author{Federico Romaniello}
\ead{federico.romaniello@unibas.it}
\address{Dipartimento di Matematica, Informatica ed Economia - Universit\`{a} degli Studi della Basilicata - Viale dell'Ateneo Lucano 10 - 85100 Potenza (Italy),}
\author{Tam\'as Sz\H{o}nyi}
\ead{szonyi@cs.elte.hu}
\address{Institute of Mathematics, E\"otv\"os University and MTA-ELTE Geometric and Algebraic Combinatorics Research Group, P\'azm\'any P. s. 1/C Budapest 1117, Hungary and FAMNIT, University of Primorska, Glagolja\v{s}ka 8, 6000 Koper, Slovenia.}


\begin{abstract} In the $m$-dimensional affine space $AG(m,q)$ over the finite
 field $\mathbb{F}_q$ of odd order $q$, the analog of the Euclidean distance
 gives rise to a graph $\mathfrak{G}_{m,q}$ where vertices are the points of
$AG(m,q)$ and two vertices are adjacent if their (formal) squared Euclidean
 distance is a square in $\mathbb{F}_q$ (including the zero). In 2009, Kurz and Meyer made the conjecture that if $m$ is even then $\mathfrak{G}_{m,q}$
is a strongly regular graph. In this paper we prove their conjecture.
\end{abstract}

\end{frontmatter}


\section{Introduction}
In Euclidean geometry, problems about integral point-sets, i.e. point-sets in
which the distances among points are integral, have a long history and still
receive much attention; see \cite[Section 5.11]{dagi}, and \cite{KK} for an
overview on the most recent results.
Finite field analogs of such problems have also been the subject of several
papers where ``integral distance'' is meant ``squared Euclidean distance to
be a square in $\mathbb{F}_q$ (including $0$)''.   Here the squared Euclidean
distance of two points ${\bf{x}}=(x_1,\ldots,x_m)$ and
${\bf{y}}=(y_1,\ldots,y_m)$ in the $m$-dimensional affine space $AG(m,q)$
over $\mathbb{F}_q$ is
$$d({\bf{x}},{\bf{y}})=\sum_{i=1}^m (x_i-y_i)^2.$$
Iosevich, Shparlinski and Xiong \cite{ISX} gave upper bounds on the size of
integral point-sets in $AG(m,q)$ and proved results on those with the largest
possible cardinality. On the other hand, finding all maximal integral point-sets
seems to be out of reach, even in $AG(2,q)$, apart from smaller $q$'s where a
complete classification may be obtained with heavy computer aided computation.
This has been done so far for $m=2$ with $q\le 47$; see \cite{KM}. Integral
automorphisms of $AG(m,q)$, i.e. permutations on the point-set of $AG(m,q)$
which take integral distances to integral distances, were thoroughly
investigated by Kurz \cite{K}, Kiermaier and Kurz \cite{KK}, Kurz and Meyer
\cite{KM}, Kov\'acs and Ruff \cite{KR}, Kov\'acs, Kutnar, Ruff and Sz\H{o}nyi
\cite{KKRSz}. The final outcome was that integral automorphisms are semilinear
transformations for $m\geq 3$, for $m=2$ and $q\equiv 3 \pmod 4$, with some
exceptions for $m=2$ and $q\equiv 1 \pmod 4$.

 Define a graph $\mathfrak{G}_{m,q}$ whose vertices are the points of $AG(m,q)$
and two vertices of $\mathfrak{G}_{m,q}$ are adjacent if their distance is
integral. As Kurz and Meyer \cite{KM} pointed out, the previous problems and
results on integral distances can be interpreted in terms of cliques and
automorphisms of $\mathfrak{G}_{m,q}$. As its automorphism group contains all
translations, the graph $\mathfrak{G}_{m,q}$ is vertex transitive and hence it
is regular. Therefore, it is natural to ask whether the graph is also strongly
regular. For $m=2$, this was proved by Kiermaier and Kurz \cite{KK}.
For $m=2$ and $q\equiv 3 \pmod 4$, the graph  $\mathfrak{G}_{m,q}$ is
isomorphic to the Paley graph. Cliques of maximal sizes of the Paley graph are
known from previous work of Blokhuis \cite{AB}. The general case $m\ge 3$ was
investigated by Kurz and Mayer \cite{KM}. Kurz and Meyer computed the value of
$\lambda$ for adjacent pairs of vertices whose distance is a non-zero
square. They proved that the graph $\mathfrak{G}_{m,q}$ is not strongly
regular, if $m$ is odd. Besides computing the value of $\lambda$ for most
adjacent pairs of vertices, they also showed some evidence, supported by
computer aided searches, that $\mathfrak{G}_{m,q}$ should be a strongly
regular graph for even $m$ with parameters $(v,k,\lambda,\mu)$ where:
\begin{equation}
\label{eq100}
\begin{array}{llllll}
v=q^m,\\
\vspace{0.2cm}
k=
\begin{cases}
{\mbox{$\textstyle\frac{1}{2}(q^m+q^{m-1}+q^{m/2}-q^{m/2-1})-1$, for $q\equiv 1 \pmod 4$}}\\
{\mbox{$\textstyle\frac{1}{2}(q^m+q^{m-1}+(-q)^{m/2}+(-q)^{m/2-1})-1,$ for $q\equiv 3 \pmod 4$}},\\
\end{cases}
\\
\vspace{0.2cm}
\lambda=
{\mbox{$\textstyle\frac{1}{4}(q^{m-2}(q+1)^2+2(-1)^{m(q-1)/4}q^{m/2-1}(q-1))-2$}},\\
\mu=
\begin{cases}
{\mbox{$\textstyle\frac{1}{4}q^{m/2-1}(q+1)(q^{m/2}+q^{m/2-1}+2)$, for $q\equiv 1 \pmod 4$}}\\
{\mbox{$\textstyle\frac{1}{4}q^{m/2-1}(q+1)(q^{m/2}+q^{m/2-1}+2(-1)^{m/2})$, for $q\equiv 3 \pmod 4$}},\\
\end{cases}
\end{array}
\end{equation}
see \cite[Conjecture 4.6]{KM}. Therefore, in this paper we will suppose that
$m\ge 4$ and $m$ is even, and prove the Kurz-Meyer conjecture in this case.

\section{Background}
A graph is strongly regular if every vertex has $k$ neighbors, every adjacent
pair of vertices has $\lambda$ common neighbors, and every pair of
non-adjacent vertices has $\mu$ neighbors. These integers together with the
number of vertices $v$ are the parameters of a strongly regular graph. They
are not independent parameters, as $(v-k-1)\mu=k(k-\lambda-1)$.
The following lemma is useful for our purpose.
\begin{lemma}
\label{lem3nov2018} Let $\Gamma$ be a regular graph of size $v$ and degree
$k$ with $v,k$ as in (\ref{eq100}). Assume that $\Gamma$ has an automorphism
group $G$ with the following properties:
\begin{itemize}
\item[\rm(i)] $G$ is vertex-transitive.
\item[\rm(ii)] The stabilizer of a vertex $O$ has exactly three non-trivial
orbits $\cP_0,\cP^+$ and $\cP^-$ where $\cP^+$ and $\cP^-$ have the same size
given by:
$$
\begin{cases}
{\mbox{$\ha(q^m-q^{m-1}-q^{m/2}+q^{m/2-1})$ for $q\equiv 1 \pmod 4$}};\\
{\mbox{$\ha(q^m-q^{m-1}-(-q)^{m/2}-(-q)^{m/2-1})$ for $q\equiv 3 \pmod 4$}};\\
\end{cases}
$$
 Assume, in addition, that
\item[\rm(iii)] For $P^+\in \cP^+$, the number of common neighbors of $O$
and $P^+$ is $\lambda$ as given in (\ref{eq100}), and that
\item[\rm(iv)]  For $P^-\in \cP^-$, the number of common neighbors of $O$
and $P^-$ is $\mu$ as given in (\ref{eq100}).
\end{itemize}
If the neighbors of $O$ are the vertices in $\cP_0\cup \cP^+$, then $\Gamma$
is a strongly regular graph with parameters $(v,k,\lambda,\mu)$.
\end{lemma}
\begin{proof} For a vertex $P_0\in \cP_0$, let $\sigma$ be the number of
common neighbors of $P_0$ and $O$. To prove the assertion, it is enough to
show $\sigma=\lambda$.
For this purpose, a variant of a well known counting argument is used.
The vertices of $\Gamma$ are arranged in three levels in the following way.
Vertex $O$ lies in Level 0. Its $k$ neighbors lie in Level 1, and all other
vertices in Level 2. Vertices in Level 1 are those in $\cP_0\cup \cP^+$, and
each with $O$ shares either $\sigma$ or $\lambda$ neighbors according as it
is in $\cP_0$ or $\cP^+$. Clearly, all these common neighbors must also be in
Level 1. Since each vertex has degree $k$, there are $k-\sigma-1$
(resp. $k-\lambda-1)$ edges remaining for each Level 1 vertex in $\cP_0$
(resp. $\cP^+$) to connect to vertices in Level 2. Therefore, there are
$|\cP_0|(k-\sigma-1)+|\cP^+|(k-\lambda-1)$ edges between Level 1 and Level 2.

Since vertices in Level 2 are not connected to $O$, they must have  common
neighbors with $O$, and these common neighbors must all be in Level 1.
There are $v-k-1$ vertices in Level 2, and each is connected to vertices in
Level 1. Therefore the number of edges between Level 1 and Level 2 is
$ (v-k-1)\mu.$

Equating the two expressions for the edges between Level 1 and Level 2 shows
that $|\cP_0|(k-\sigma-1)+|\cP^+|(k-\lambda-1)= (v-k-1)\mu$. Therefore,
$\sigma$ is uniquely determined by $v,\lambda,\mu$. A straightforward
computation (or, the proof of \cite[Lemma 4.5]{KM}) shows that $\sigma=\lambda$.
\end{proof}
As we have mentioned in the Introduction, the automorphism group ${\rm{Aut}}(\mathfrak{G}_{m,q})$ of the graph
$\mathfrak{G}_{m,q}$ is known. From the remark prior to \cite[Lemma 4.1]{KM}, ${\rm{Aut}}(\mathfrak{G}_{m,q})$ acts transitively on the vertices of $\mathfrak{G}_{m,q}$. Also,
\cite[Theorem 3.2]{KM} and \cite[Lemma 3.17]{KM} show that the non-trivial orbits of the stabilizer of the zero-vector under the action of ${\rm{Aut}}(\mathfrak{G}_{m,q})$ are three, namely $\cP^+$, $\cP^-$ and $\cP_0$.
Here, $|\cP^+|=|\cP^-|$ is the same number as in the formula given in (ii) of Lemma \ref{lem3nov2018}; see \cite[Theorem 4.3]{KM}.
\begin{lemma}
\label{lem4nov2018} The automorphism group of $\mathfrak{G}_{m,q}$ fulfills
Hypotheses (i) and (ii) of Lemma \ref{lem3nov2018}.
\end{lemma}
A major result in \cite{KM} states that (iii) of Lemma  \ref{lem3nov2018} holds for $\mathfrak{G}_{m,q}$; see \cite[Theorem 4.4]{KM}. Therefore, the Kurz-Meyer conjecture \cite[Conjecture 4.6]{KM} is true if and only if  (iv) of Lemma  \ref{lem3nov2018} holds for $\mathfrak{G}_{m,q}$. This conjecture is known to be true for $m=2$, see \cite[Section 3]{KK}, and also  for some small values of $q=p$ prime, namely for $(m = 3, p\le2029), (m = 4, p \le 283), (m = 5, p\le 97), (m = 6, p\le 59), (m = 7, p\le 31)$,  and $(m = 8, p\le 23)$; see  \cite[Section 4]{KM}.

Our proof of the  Kurz-Meyer conjecture is performed in geometric terms in
the $m$-dimensional finite projective space over a finite field $\mathbb{F}_q$
of odd order $q$. Accordingly, our notation and terminology are those in
\cite{ht} with one exception: In $\mathbb{F}_q$, half of the non-zero
elements are squares and half are non-squares. In this paper, the former
set is denoted by $\Box$ and the latter by $\not\hspace{-0.1cm}\Box$.
Furthermore, we fix an element $\varepsilon \in \mathbb{F}_q$ such that
$$1+\varepsilon^2\in \not\hspace{-0.1cm}\Box.$$

The following lemma on quadrics in $PG(r,q)$ is frequently used in our proofs.
\begin{lemma}
({\rm{\cite[Theorem 22.2.1 (ii), Theorem 22.5.1(b)]{ht}}})
In the $r$-dimensional projective space $PG(r,q)$ over $\mathbb{F}_q$, any
non-singular quadric  is either parabolic, or elliptic, or hyperbolic. More
precisely,
\begin{itemize}
\label{quadrics}
\item[\rm(i)] for even $r$, any non-singular quadric is parabolic.
\item[\rm(ii)] for odd $r$, any non-singular quadric is either elliptic or
hyperbolic according as $(-1)^{(r+1)/2}d\in\not\hspace{-0.1cm}\Box$ or
$(-1)^{(r+1)/2}d\in\Box$ where $d$ is the determinant of the associated matrix.
\end{itemize}
Moreover, the number of points of a non-singular quadric in $PG(r,q)$ is:
$$\begin{cases}
{\mbox{$(q^{r/2}+1)(q^{r/2}-1)/(q-1)$ for a parabolic quadric}},\\
{\mbox{$(q^{(r+1)/2}+1)(q^{(r-1)/2}-1)/(q-1)$ for an elliptic quadric}},\\
{\mbox{$(q^{(r-1)/2}+1)(q^{(r+1)/2}-1)/(q-1)$ for a hyperbolic quadric}}.
\end{cases}
$$
\end{lemma}
\section{Reformulation of the Conjecture and Outline of the Proof}
\label{reform}
The number of common neighbours of the vertices represented by $(0,0,\ldots,0)$ and $(1,\varepsilon,0,\ldots,0)$ is equal to the number of ordered $m$-tuples $(y_1,\ldots,y_m)$
such that
\begin{equation}
\label{eq1} y_1^2+y_2^2+\ldots+y_m^2 \in \Box \cup \{0\},
\end{equation}
and
\begin{equation}
\label{eq2} (y_1-1)^2+(y_2-\varepsilon)^2+y_3^2+\ldots+y_m^2 \in \Box \cup \{0\}.
\end{equation}
By Lemmas \ref{lem3nov2018} and \ref{lem4nov2018}, to prove the Kurz-Meyer
conjecture it is enough to prove the following theorem.
\begin{theorem}
\label{pr1} Given an even integer $m\geq 4$, the number $\mu$ as given in
(\ref{eq100}) is equal to the number of ordered $m$-tuples $(y_1,\ldots,y_m)$
such that both (\ref{eq1}) and (\ref{eq2}) hold.
\end{theorem}
We outline of the proof of Theorem \ref{pr1}.
Equation (\ref{eq2}) means that there exists $\beta\in \mathbb{F}_q$ such that
\begin{equation}
\label{15octeq1}
\beta^2= (y_1-1)^2+(y_2-\varepsilon)^2+y_3^2+\ldots+y_m^2.
\end{equation}
It is convenient to rewrite Equation (\ref{eq1}) asking the existence
of $\gamma\in \mathbb{F}_q$ such that
\begin{equation}
\label{15octeq2}
(\gamma+\beta)^2=y_1^2+y_2^2+\ldots+y_m^2.
\end{equation}

If we have a typical solution $(\beta,\gamma)$ of this system of two quadratic
equations, then so are $(\beta, -2\beta-\gamma)$ and
$(-\beta,-\gamma)$, $(-\beta, 2\beta+\gamma)$. This already shows that the
number of solutions will not depend on the values of $\beta$ and $\gamma$ but
rather on $\beta^2$ and $\gamma^2$. With this, a typical solution will
correspond to one $\beta^2$ and two values of $\gamma ^2$.

Subtracting (\ref{15octeq1}) from (\ref{15octeq2}) gives a linear equation
in $y_1$ and $y_2$, namely
\begin{equation}
\label{eq2-eq1}
2\beta\gamma +\gamma^2=-(1+\varepsilon^2)+2y_1+2\varepsilon y_2.
\end{equation}
For $\gamma =0$ we have $y_1=\frac{1}{2}(1+\varepsilon ^2)-\varepsilon
y_2$. Substituting this in (\ref{15octeq1}) gives
\begin{equation}
\label{eq12bis} (1+\varepsilon^2)y_2^2-\varepsilon(1+\varepsilon^2)y_2+
\textstyle\frac{1}{4}(1+\varepsilon^2)^2-\beta^2+\sum_{i=3}^my_i^2=0,
\end{equation}
for some $\beta\in \mathbb{F}_q$. We will determine the number of solutions
$(y_2,\ldots ,y_m)$ of (\ref{eq12bis}) in Section \ref{PrA}. Note that whenever
we have a solution $(y_2,\ldots ,y_m)$ to Equation (\ref{eq12bis}) for a certain $\beta^2$, then
by putting  $y_1=\frac{1}{2}(1+\varepsilon ^2)-\varepsilon y_2$ we can uniquely
complete $(y_2,\ldots ,y_m)$ into a solution $(y_1,y_2,\ldots,y_m)$ of the system of equations (\ref{15octeq1}) and (\ref{15octeq2}) with $\gamma=0$.

In the case where $\gamma^2\not= 0$, we get $\beta=\ha (2y_1+2\varepsilon
y_2-(1+\varepsilon^2+\gamma^2))/\gamma$. Substituting this into Equation
(\ref{15octeq1}) the following equation is obtained:
\begin{equation}
\label{eq3}
(y_1+\varepsilon y_2-\ha(1+\varepsilon^2-\gamma^2))^2-\gamma^2 \sum_{i=1}^my_i^2=0.
\end{equation}
Note that whenever we have a solution of (\ref{eq3}) for a fixed $\gamma$,
then by putting $\beta=\ha (2y_1+2\varepsilon y_2-(1+\varepsilon^2+
\gamma^2))/\gamma$ we see that this is a solution of the system
(\ref{15octeq1}) and (\ref{15octeq2}) and vice versa. So we will determine
the number of solutions of (\ref{eq3}) for a fixed $\gamma^2$, denoted
by $[\gamma^2]$. Note that although the fixed $\gamma^2$ arises from two values,
namely $\pm \gamma$, they will give, however, the same $\beta^2$.

Actually, the symbol $[0]$ will also be useful. Since, for $\gamma=0$,
(\ref{eq3}) reads $y_1=\frac{1}{2}(1+\varepsilon ^2)-\varepsilon y_2$,
the appropriate meaning of $[0]$ is the number of solutions of (\ref{eq12bis}),
and it can be denoted as $[0]$.

In Section \ref{PrB} we will compute $[\gamma^2]$ or rather
$[0]+\sum[\gamma^2]$. Since every typical solution $(y_1,\ldots ,y_m)$ of the
system (\ref{15octeq1}) and (\ref{15octeq2}) occurs for two values of
$\gamma^2$, we have to determine the (number of) solutions which belong
to exactly one $\gamma^2$, and finally use this piece of information to
determine the number of solutions $(y_1,\ldots ,y_m)$ of the system
(\ref{15octeq1}) and (\ref{15octeq2}), without their multiplicity in
$[0]+\sum[\gamma^2]$. This number has to be the conjectured parameter
$\mu$ of our graph.

From the above discussion one can also see what the non-typical values of
$\gamma^2$ are. They satisfy either $\gamma =-2\beta -\gamma$ or
$-\gamma =-2\beta -\gamma$. The former case occurs when $\gamma=-\beta$, the
latter when $\beta=0$. So, these cases correspond to solutions of the system
(\ref{15octeq1}) and (\ref{15octeq2}), where either $\beta$ or
$\beta +\gamma$ (that is, the left hand side of one of the two equations)
is $0$. This includes the doubly special case, $\beta=\gamma=0$, which we
treat in Section \ref{PrA}.

\section{Case $\gamma =0$}
\label{PrA}

In this section we are going to consider the case $\gamma=0$.

\begin{proposition}
\label{prA} Given an even integer $m\geq 4$, the number $\sigma=[0]$ of ordered
$m$-tuples $(y_1,y_2,\ldots,y_m)$ with
$y_1=\ha (1+\varepsilon^2)-\varepsilon y_2$ such that (this is equation
(\ref{eq12bis}))
\begin{equation*}
(1+\varepsilon^2)y_2^2-\varepsilon(1+\varepsilon^2)y_2+
\textstyle\frac{1}{4}(1+\varepsilon^2)^2-\beta^2+\sum_{i=3}^my_i^2=0,
\end{equation*}
for some $\beta\in \mathbb{F}_q$ is equal to:
$$
\begin{cases}{\mbox{$\ha q^{m/2-1}(q^{m/2}+q^{m/2-1}+2)$,\,\,\, for $m\equiv
0 \hspace{-0.2cm} \pmod 4$,\,\, $q\equiv 1 \hspace{-0.2cm}\pmod 4$}},\\
{\mbox{$\ha q^{m/2-1}(q^{m/2}+q^{m/2-1})$,\,\,\,\,\,\,\,\,\,\,\,\,\, for
$m\equiv 0 \hspace{-0.2cm} \pmod 4$,\,\, $q\equiv 3 \hspace{-0.2cm}\pmod 4$}},\\
{\mbox{$\ha q^{m/2-1}(q^{m/2}+q^{m/2-1}+2)$,\,\,\, for $m\equiv 2
\hspace{-0.2cm} \pmod 4$,\,\, $q\equiv 1 \hspace{-0.2cm}\pmod 4$}},\\
{\mbox{$\ha q^{m/2-1}(q^{m/2}+q^{m/2-1})$,\,\,\,\,\,\,\,\,\,\,\,\,\, for
$m\equiv 2 \hspace{-0.2cm} \pmod 4$,\,\, $q\equiv 3 \hspace{-0.2cm}\pmod 4$}}.\\
\end{cases}
$$

\end{proposition}

\begin{proof}
 Equation (\ref{eq12bis}) can be viewed as the (affine) equation of a
quadric $Q$ in the $m$ dimensional projective space $PG(m,q)$ with coordinates
$(\beta,y_2,\ldots,y_m,t)$ where the homogeneous equation of $Q$ is
\begin{equation}
\label{eq12ter}
\beta^2-(1+\varepsilon^2)y_2^2+\varepsilon(1+\varepsilon^2)y_2t
-\textstyle\frac{1}{4}(1+\varepsilon^2)^2t^2-\sum_{i=3}^my_i^2=0.
\end{equation}
The matrix associated with $Q$ is not singular as its determinant equals
$\frac{1}{4}(1+\varepsilon^2)^2$. As $m$ is even, $Q$ is a parabolic quadric
and it has as many as $N_q=(q^{m/2}+1)(q^{m/2}-1)/(q-1)$
points; see Lemma \ref{quadrics}. Now, compute the number $n_q$ of points of
$Q$ at infinity. They are the points of the quadric $S$ in $PG(m-1,q)$ with
homogeneous equation
\begin{equation}
\label{eq12for} \beta^2-(1+\varepsilon^2)y_2^2-\sum_{i=3}^my_i^2=0.
\end{equation}
The matrix associated with $S$ has determinant $-(1+\varepsilon^2)$ that does
not vanish. More precisely, $-(1+\varepsilon^2)\in\Box$ or $-(1+\varepsilon^2)
\in \not\hspace{-0.1cm}\Box$ depending upon whether $q\equiv 3 \pmod 4$, or
$q\equiv 1 \pmod 4$. Therefore, $S$ is hyperbolic for $q\equiv 3 \pmod 4$ and
$m\equiv 0 \pmod 4$, otherwise it is elliptic. Let $n_q$ be the number of
points of $S$. Then, from Lemma \ref{quadrics}, if $m\equiv 0 \pmod 4$ then
$$n_q=
\begin{cases}
{\mbox{$(q^{{m/2}-1}-1)(q^{m/2}+1)/(q-1)$\,\, for $q\equiv 1 \pmod 4$}},\\
{\mbox{$(q^{m/2}-1)(q^{m/2-1}+1)/(q-1)$\,\, for $q\equiv 3 \pmod 4$}},
\end{cases}
$$
and if $m\equiv 2 \pmod 4$ then
$$n_q=(q^{m/2-1}-1)(q^{m/2}+1)/(q-1).$$

Therefore,
$$N_q-n_q=
\begin{cases}{\mbox{$q^{m/2-1}(q^{m/2}+1)$ for $m\equiv 0 \pmod 4$\,\,
and $q\equiv 1 \pmod 4$}},\\
{\mbox{$q^{m/2-1}(q^{m/2}-1)$ for $m\equiv 0 \pmod 4$\,\,
and $q\equiv 3 \pmod 4$}},\\
{\mbox{$q^{m/2-1}(q^{m/2}+1)$ for $m\equiv 2 \pmod 4$\,\,
and $q\equiv 1 \pmod 4$}},\\
{\mbox{$q^{m/2-1}(q^{m/2}+1)$ for $m\equiv 2 \pmod 4$\,\,
and $q\equiv 3 \pmod 4$}}.
\end{cases}
$$
The points $(\beta,y_2,\ldots,y_m)$ of $Q$ are of two types, according as
$\beta=0$, or $\beta\neq 0$.
To count the points of $Q$ with $\beta=0$, consider the quadric $Q'$ of
in the $m-1$ dimensional projective space $PG(m-1,q)$ with coordinates
$(y_2,\ldots,y_m,t)$, whose (affine) equation is
\begin{equation}
\label{eq22ter} -(1+\varepsilon^2)y_2^2+\varepsilon(1+\varepsilon^2)y_2-
\textstyle\frac{1}{4}(1+\varepsilon^2)^2+\sum_{i=3}^my_i^2=0.
\end{equation}

The homogeneous equation of $Q'$ is
\begin{equation}
\label{eq12Ater} -(1+\varepsilon^2)y_2^2+\varepsilon(1+\varepsilon^2)y_2t-
\textstyle\frac{1}{4}(1+\varepsilon^2)^2t^2+\sum_{i=3}^my_i^2=0.
\end{equation}

The matrix associated with $Q'$ has determinant
$\frac{1}{4}(1+\varepsilon^2)^2$ which falls in\, $\Box$. Thus, $Q'$ is
elliptic for $q\equiv 3 \pmod 4$ and $m\equiv 2 \pmod 4$, otherwise it is
hyperbolic. The points of $Q'$ at infinity are those of the quadric $S'$ of
$PG(m-2,q)$ with homogeneous equation
\begin{equation}
\label{eq12Afor} -(1+\varepsilon^2)y_2^2-\sum_{i=3}^my_i^2=0.
\end{equation}
The matrix associated with $S'$ has determinant $-(1+\varepsilon^2)$
that does not vanish. Therefore, $S'$ is a parabolic quadric and it has
$(q^{m/2-1}+1)(q^{m/2-1}-1)/(q-1)$ points. Let $\sigma_0$ be the number of
points of  $Q'$. From Lemma \ref{quadrics}, if $m\equiv 0 \pmod 4$ then
$\sigma_0=[(q^{m/2}-1)(q^{m/2-1}+1)-(q^{m/2-1}+1)(q^{m/2-1}-1)]/(q-1)=
q^{m/2-1}(q^{m/2-1}+1)$. Hence, if $m\equiv 0 \pmod 4$, then
$$\sigma_0=q^{m/2-1}(q^{m/2-1}+1).$$
Similar computation shows that if $m\equiv 2 \pmod 4$, then
$$\sigma_0=
\begin{cases}
{\mbox{$q^{m/2-1}(q^{m/2-1}+1)$\,\, for $q\equiv 1 \pmod 4$}},\\
{\mbox{$q^{m/2-1}(q^{m/2-1}-1)$\,\, for $q\equiv 3 \pmod 4$}}.
\end{cases}
$$
$[0]$ is the number of affine points of $Q$ with $\beta =0$
  plus half of the number of affine points of $Q$ with $\beta \not= 0$,
hence $[0]=\sigma=\ha(N_q-n_q+\sigma_0)$, and the assertion follows.
\end{proof}

Let us remark that $\sigma_0$ is the number of those $y_1,\ldots ,y_m$, for
which $\beta=\gamma=0$ holds.

\section{Computation of $[\gamma^2]$ and their sum}
\label{PrB}
In this section, $\gamma$ is always a non-zero element in $\mathbb{F}_q$.
Also, for shortness, we use the symbol ${\bf{y}}$ to denote any ordered
$m$-tuple $(y_1,y_2,\ldots,y_m)$ with $y_i \in \mathbb{F}_q$. We first want
to compute the number of ${\bf{y}}$ satisfying Equation (\ref{eq3}) for a
fixed $\gamma^2$. This number was denoted by
$[\gamma^2]$ in Section \ref{reform}.
 It should be noticed that $[0]$ denoted the number of affine points of $Q$ with $\beta=0$ plus half the number of affine points $Q$ with $\beta\neq 0$
As in Section \ref{PrA},
Equation (\ref{eq3}) is viewed as the (affine) equation of a quadric
$Q_\gamma$. So, we first homogenize the equation and consider the quadric
in $PG(m,q)$ with variables $(y_1,\ldots y_m,t)$. This homogenized equation is
$$ y_1^2+\varepsilon^2 y_2^2+2\varepsilon y_1y_2-y_1t(1+\varepsilon^2-\gamma^2)-y_2t\varepsilon
(1+\varepsilon^2-\gamma^2)+t^2\frac{1}{4}(1+\varepsilon^2-\gamma^2)^2-\gamma^2
\sum_{i=1}^my_i^2=0.
$$
The matrix $M_\gamma$ of $Q_\gamma$ has (at most) the following non-zero
elements:
\begin{equation}
\label{eq20oct}
\begin{cases}
a_{11}=(1-\gamma^2),a_{12}=a_{21}=\varepsilon,
a_{22}=(\varepsilon^2-\gamma^2),
a_{33}=\ldots=a_{m,m}=-\gamma^2,\\
a_{1,m+1}=a_{m+1,1}=-\frac{1}{2}(1+\varepsilon^2-\gamma^2),
a_{2,m+1}=a_{m+1,2}= -\frac{1}{2}\varepsilon(1+\varepsilon^2-\gamma^2),\\
a_{m+1,m+1}=\frac{1}{4}(1+\varepsilon^2-\gamma^2)^2.
\end{cases}
\end{equation}
A straightforward computation gives
$$\det(M_\gamma)=\frac{1}{4}(-\gamma^2)^{m}(1+\varepsilon^2-\gamma^2)^2,$$
which shows that $\det(M_\gamma)\in \Box$.
Let $g_q=[\gamma^2]$.
\begin{lemma}
\label{lem20octA}  If $m\equiv 0 \pmod 4$, then
\begin{equation}
\label{eq4} g_q=
\begin{cases} {\mbox{$q^{(m/2)-1}(q^{m/2}-1)$ for $\gamma^2-(1+\varepsilon^2)
\in\Box$,}}\\
 {\mbox{$q^{(m/2)-1}(q^{m/2}+1)$ for $\gamma^2-(1+\varepsilon^2)\in\,
\not\hspace{-0.1cm}\Box$.}}
 \end{cases}
\end{equation}
If $m\equiv 2 \pmod 4$ and $q\equiv 1 \pmod 4$, then
\begin{equation}
\label{eq4bis} g_q=
\begin{cases} {\mbox{$q^{m/2-1}(q^{m/2}-1)$ for $\gamma^2-(1+\varepsilon^2)
\in\Box,$}}\\
 {\mbox{$q^{m/2-1}(q^{m/2}+1)$ for $\gamma^2-(1+\varepsilon^2)\in\,
\not\hspace{-0.1cm}\Box$.}}
 \end{cases}
\end{equation}
If $m\equiv 2 \pmod 4$ and $q\equiv 3 \pmod 4$, then
\begin{equation}
\label{eq4bisbis} g_q=
\begin{cases}  {\mbox{$q^{m/2-1}(q^{m/2}+1)$ for $\gamma^2-(1+\varepsilon^2)
\in\Box $,}}\\
{\mbox{$q^{m/2-1}(q^{m/2}-1)$ for $\gamma^2-(1+\varepsilon^2)\in\,
\not\hspace{-0.1cm}\Box$.}}
 \end{cases}
\end{equation}
\end{lemma}
\begin{proof} We argue as in Section \ref{PrA}. Since $\det(M_\gamma)\in \Box$,
in the $m$-dimensional projective space $PG(m,q)$ the (non-degenerate,
parabolic) quadric $Q_\gamma$ has
$$N_q=(q^{m/2}+1)(q^{m/2}-1)/(q-1)$$
points; see Lemma \ref{quadrics}.
Next we find the number $t_q$ of points of $Q_\gamma$ at infinity by
substituting $t=0$. Such points of $Q_\gamma$ are those of the quadric
$S_\gamma$ in $PG(m-1,q)$ with homogeneous equation
 \begin{equation}
\label{20octeq3}
(1-\gamma^2)y_1^2+2\varepsilon y_1y_2+(\varepsilon^2-\gamma^2)y_2^2-
\gamma^2\sum_{i=3}^my_i^2=0.
\end{equation}
The associated matrix has (at most) the following non-zero elements:
$$
b_{11}=(1-\gamma^2),b_{12}=b_{21}=\varepsilon, b_{22}=(\varepsilon^2-\gamma^2),
b_{33}=\ldots=b_{m,m}=-\gamma^2.
$$
and its determinant does not vanish being equal to
$\gamma^2(\gamma^2-(1+\varepsilon^2))(-\gamma^2)^{m-2}$.

Two cases are distinguished according as  $m\equiv 0 \pmod 4$ or
$m\equiv 2 \pmod 4$. In the former case, $S_\gamma$ is hyperbolic for
$\gamma^2-(1+\varepsilon^2)\in \Box$ and elliptic for
$\gamma^2-(1+\varepsilon^2)\in\not\hspace{-0.1cm}\Box$. Hence, if
$m\equiv 0 \pmod 4$ then
$$t_q=
\begin{cases}
{\mbox{$(q^{{m/2}-1}+1)(q^{m/2}-1)/(q-1)$ for $\gamma^2-(1+\varepsilon^2)
\in\Box$}},\\
{\mbox{$(q^{m/2}+1)(q^{{m/2}-1}-1)/(q-1)$ for $\gamma^2-(1+\varepsilon^2)
\in\,\not\hspace{-0.1cm} \Box$}}.
\end{cases}
$$
In the latter case, if $q\equiv 1 \pmod 4$ then $S_\gamma$ is hyperbolic
for $\gamma^2-(1+\varepsilon^2)\in \Box$ and is elliptic for $\gamma^2-
(1+\varepsilon^2)\in\not\hspace{-0.1cm} \Box$. If  $q\equiv 3 \pmod 4$,
then the same holds provided that the adjectives hyperbolic and elliptic
are interchanged.
Therefore, if $m\equiv 2 \pmod 4$ and $q\equiv 1 \pmod 4$ then
$$t_q=
\begin{cases}
{\mbox{$(q^{{m/2}-1}+1)(q^{m/2}-1)/(q-1)$ for $\gamma^2-(1+\varepsilon^2)
\in\Box$}},\\
{\mbox{$(q^{m/2}+1)(q^{{m/2}-1}-1)/(q-1)$ for $\gamma^2-(1+\varepsilon^2)
\in\,\not\hspace{-0.1cm} \Box$}}
\end{cases}
$$
but if $m\equiv 2 \pmod 4$ and $q\equiv 3 \pmod 4$ then
$$t_q=
\begin{cases}
{\mbox{$(q^{m/2}+1)(q^{{m/2}-1}-1)/(q-1)$ for $\gamma^2-(1+\varepsilon^2)
\in \Box$}},\\
{\mbox{$(q^{{m/2-1}}+1)(q^{m/2}-1)/(q-1)$ for $\gamma^2-(1+\varepsilon^2)
\in\,\not\hspace{-0.1cm}\Box$}}.
\end{cases}
$$
Since $g_{q}=[\gamma^2]=N_q-t_q$, the assertion follows.
\end{proof}

The next lemma helps us determine how many times $\gamma^2-(1+\varepsilon^2)$
is a square and a non-square.

\begin{lemma}
\label{lemC14oct} For a fixed non-zero element $s\in \mathbb{F}_q$, let
$R_q$ be the number of $\gamma^2$ for which  the equation
$$\gamma^2-s\tau^2=1+\varepsilon^2$$
has a solution $0\not= \gamma\in \mathbb{F}_q$ for some $\tau\in \mathbb{F}_q$. Then
\begin{equation}
\label{20octeq41}
R_q=
\begin{cases}{\mbox{$\frac{1}{4}(q-1)$ for $q\equiv 1 \pmod 4 $,}}\\
{\mbox{$\frac{1}{4}(q-3)$ for $s\in \Box$ and $q\equiv 3 \pmod 4$}},\\
{\mbox{$\frac{1}{4}(q+1)$ for $s\in\,\not\hspace{-0.1cm} \Box$ and
$q\equiv 3 \pmod 4$}}.
\end{cases}
\end{equation}
\end{lemma}
\begin{proof} In $AG(2,q)$ with coordinates $(X,Y)$, the equation
$X^2-sY^2=1+\varepsilon^2$ defines an (irreducible) conic $\cC$ which is
a hyperbola or an ellipse according as $s\in \Box$ or
$s\in\,\not\hspace{-0.1cm} \Box$. In the hyperbola case, $\cC$ has
$q-1$ points. Moreover, either $0$ or $2$ of these points lie on the line
$X=0$ according as $q\equiv 1 \pmod 4$ or $q\equiv 3 \pmod 4$. Therefore,
if $s\in\Box$ then either $R_q=\frac{1}{4}(q-1)$ or $R_q=\frac{1}{4}(q-3)$
according as $q\equiv 1 \pmod 4$ or $q\equiv 3 \pmod 4$.
In the elliptic case, $\cC$ has as many as $q+1$ points, and either $2$ or
$0$ of them lie on the line $X=0$ depending on $q\equiv 1 \pmod 4$ or
$q\equiv 3 \pmod 4$. Therefore, if $s\in\,\not\hspace{-0.1cm} \Box$ then
either $R_q=\frac{1}{4}(q-1)$ or $R_q=\frac{1}{4}(q+1)$ according as
$q\equiv 1 \pmod 4$ or $q\equiv 3 \pmod 4$.
\end{proof}
Lemmas \ref{lem20octA} and \ref{lemC14oct} have the following corollary.
It can be seen immediately by splitting the sum $\sum [\gamma^2]$ into
two subsums
$$\sum [\gamma^2]=\sum_{\gamma^2-(1+\varepsilon^2)\in \Box} [\gamma^2]+
\sum_{\gamma^2-(1+\varepsilon^2)\in \not\Box} [\gamma^2].
$$
\begin{lemma}
\label{lemoct20}
$$\sum [\gamma^2]=
\begin{cases}
{\mbox{$\ha q^{m-1}(q-1)$ for $m\equiv 0 \hspace{-0.25cm}\pmod 4$ and
$q\equiv 1 \hspace{-0.25cm} \pmod 4$,}}\\
{\mbox{$\ha q^{m/2-1}(q^{m/2+1}-q^{m/2}+2)$ for $m\equiv 0 \hspace{-0.25cm}
\pmod 4$ and $q\equiv 3 \hspace{-0.25cm}\pmod 4$,}}\\
{\mbox{$\ha q^{m-1}(q-1)$ for $m\equiv 2 \hspace{-0.25cm} \pmod 4$ and
$q\equiv 1 \hspace{-0.25cm} \pmod 4$,}}\\
{\mbox{$\ha q^{m/2-1}(q^{m/2+1}-q^{m/2}-2)$ for $m\equiv 2 \hspace{-0.25cm}
\pmod 4$ and $q\equiv 3 \hspace{-0.25cm} \pmod 4$.}}
\end{cases}
$$
\end{lemma}

\section{Special values of $\gamma^2$}
\label{spec}

Our next step is to compute the number of all ${\bf{y}}$ such that (\ref{eq3})
vanishes for exactly one $\gamma^2$. As we have seen in Section \ref{reform},
the special cases arise for $\beta=0$ and $\beta+\gamma=0$. Otherwise, for
every $\beta^2$ there are exactly two $\gamma^2$'s. Note that one of the two
$\gamma^2$'s in this case can be $0$.  As the case $\gamma^2=0$
was already considered in Section \ref{PrA}, we restrict our attention in this
section to $\gamma^2\not= 0$.

\subsection{Case of $y_1^2+y_2^2+\ldots+y_m^2=0$ (that is $\beta+\gamma=0$).}
\label{case1}
In this case, from (\ref{eq3}), $y_1=\ha
(1+\varepsilon^2-\gamma^2)-\varepsilon y_2$ follows. Substituting this into
$y_1^2+y_2^2+\ldots+y_m^2=0$ gives
\begin{equation}
\label{eq8} (1+\varepsilon^2)y_2^2+y_3^2+\ldots + y_m^2-\varepsilon(1+\varepsilon^2-\gamma^2)y_2+\textstyle\frac{1}{4}(1+\varepsilon^2-\gamma^2)^2=0.
\end{equation}
The converse also holds, if $(y_2,\ldots,y_m)$ is a solution of (\ref{eq8})
and $y_1=\ha (1+\varepsilon^2-\gamma^2)-\varepsilon y_2$, then
$\beta+\gamma =0$ and (\ref{eq3}) vanishes for just one $\gamma^2$. 
\begin{lemma}
\label{13Aoct2018} Let $r_q$ be the number of all ${\bf{y}}$ with
$y_1^2+y_2^2+\ldots +y_m^2=0$ such that (\ref{eq3}) vanishes for
exactly one $\gamma^2$. Then
$$r_q=
\begin{cases} {\mbox{$\ha (q-1)q^{m/2-1}(q^{m/2-1}+1)$ for
$m\equiv 0 \pmod 4$}},\\
{\mbox{$\ha (q-1) q^{m/2-1}(q^{m/2-1}+1)$ for $m\equiv 2 \pmod 4$ and
$q\equiv 1 \pmod 4$}},\\
{\mbox{$\ha(q-1) q^{m/2-1}(q^{m/2-1}-1)$ for $m\equiv 2 \pmod 4$ and
$q\equiv 3 \pmod 4$}}.
\end{cases}
$$
\end{lemma}
\begin{proof} We look at (\ref{eq8}) as the equation of the (affine)
quadric $Q_\gamma$ of $AG(m-1,q)$, with affine coordinates $y_2,\ldots,y_m$
(and parameter $\gamma^2$), whose matrix $B_\gamma$ has (at most) the
following non-zero elements:
\begin{equation*}
b_{11}=(1+\varepsilon^2),b_{22}=\ldots=b_{m-1,m-1}=1,
b_{1,m}=b_{m,1}=-\textstyle\frac{1}{2}\varepsilon(1+\varepsilon^2-\gamma^2), \\
b_{m,m}=\textstyle\frac{1}{4}(1+\varepsilon^2-\gamma^2)^2.
\end{equation*}
A straightforward computation shows that
$$\det(B_\gamma)=\textstyle\frac{1}{4}(1+\varepsilon^2-\gamma^2)^2$$
which is in $\Box.$ As $m-1$ is odd, $Q_\gamma$ is an elliptic quadric
for $m\equiv 2 \pmod 4$ and $q\equiv 3 \pmod 4$, otherwise it is
hyperbolic. If $m_\gamma$ is the number of points of $Q_\gamma$ in
$PG(m-1,q)$, then, by Lemma \ref{quadrics},
$$ m_\gamma=
\begin{cases} {\mbox{$(q^{m/2}-1)(q^{m/2-1}+1)/(q-1)$ for
$m\equiv 0 \pmod 4$}},\\
{\mbox{$(q^{m/2}-1)(q^{m/2-1}+1)/(q-1)$ for $m\equiv 2 \pmod 4$ and
$q\equiv 1 \pmod 4$}},\\
{\mbox{$(q^{m/2}+1)(q^{m/2-1}-1)/(q-1)$ for $m\equiv 2 \pmod 4$ and
$q\equiv 3 \pmod 4$}}.
\end{cases}
$$
Now, we count the points of $Q_\gamma$ at infinity. They are the points
$(y_2,\ldots,y_m)$ satisfying the quadratic equation
 $$(1+\varepsilon^2)y_2^2+y_3^2+\ldots +y_m^2=0.$$
This equation defines a quadric $Q_\gamma'$ of
$PG(m-2,q)$ whose matrix has non-zero determinant $(1+\varepsilon^2)$.
Since $m-2$ is even, $Q_\gamma'$ is a parabolic quadric and its number
of points is equal to $t_0=(q^{m/2-1}+1)(q^{m/2-1}-1)/(q-1)$. Therefore,
$Q_\gamma$ has exactly  $M_\gamma=m_\gamma-t_0$ affine points in $AG(m-1,q)$
where
 $$M_\gamma=
\begin{cases} {\mbox{$q^{m/2-1}(q^{m/2-1}+1)$ for $m\equiv 0 \pmod 4$}},\\
{\mbox{$q^{m/2-1}(q^{m/2-1}+1)$ for $m\equiv 2 \pmod 4$ and $q\equiv 1
\pmod 4$}},\\
{\mbox{$q^{m/2-1}(q^{m/2-1}-1)$ for $m\equiv 2 \pmod 4$ and
$q\equiv 3 \pmod 4$}}.
\end{cases}
$$
Since $\gamma$ is supposed to be non-zero in this section, if $\gamma^2$ ranges in $\Box$, we  obtain $r_q=\frac{1}{2}(q-1)M_\gamma$, and
the assertion follows.
 \end{proof}
\subsection{Case of $y_1^2+y_2^2+\ldots+y_m^2=\gamma^2$ (that is $\beta=0$).}
In this case, (\ref{15octeq2}) implies $\beta =0$, hence (\ref{15octeq1}) becomes
$-2y_1-2\varepsilon y_2+(1+\varepsilon^2)+y_1^2+\ldots + y_m^2=0$.
Combining the two equations yields
\begin{equation}
\label{eq28oct} y_1=\ha (1+\varepsilon^2+\gamma^2)-\varepsilon y_2.
\end{equation}
Eliminating $y_1$ from (\ref{eq28oct}) and
$y_1^2+y_2^2+\ldots+y_m^2-\gamma^2=0$ gives
\begin{equation}
\label{eq10} (1+\varepsilon^2)y_2^2+y_3^2+\ldots y_m^2-\varepsilon(1+\varepsilon^2+\gamma^2)y_2+ \textstyle\frac{1}{4}(1+\varepsilon^2+\gamma^2)^2-\gamma^2=0.
\end{equation}
The converse also holds, if $(y_2,\ldots,y_m)$ is a solution of (\ref{eq10})
and $y_1=\ha (1+\varepsilon^2+\gamma^2)-\varepsilon y_2$, then $\beta=0$ and
(\ref{eq3})
vanishes for exactly one $\gamma^2$.
\begin{lemma}
\label{13Boct2018} For a non-zero element $\gamma\in \mathbb{F}_q$, let
$\ell_q$ be the number of all ${\bf{y}}$ with
$y_1^2+y_2^2+\ldots +y_m^2=\gamma^2$ such that (\ref{eq3}) vanishes only
for $\gamma^2$. Then
$$\ell_q=
\begin{cases} {\mbox{$\ha (q-1)q^{m/2-1}(q^{m/2-1}+1)$ for $m\equiv 0
\pmod 4$}},\\
{\mbox{$\ha (q-1) q^{m/2-1}(q^{m/2-1}+1)$ for $m\equiv 2 \pmod 4$ and
$q\equiv 1 \pmod 4$}},\\
{\mbox{$\ha(q-1) q^{m/2-1}(q^{m/2-1}-1)$ for $m\equiv 2 \pmod 4$ and
$q\equiv 3 \pmod 4$}}.
\end{cases}
$$
\end{lemma}
\begin{proof} We look at (\ref{eq10}) as the equation of the (affine)
quadric $Q_\gamma$ of $AG(m-1,q)$, with affine coordinates $y_2,\ldots,y_m$
(and parameter $\gamma^2$), whose matrix $C_\gamma$ has (at most) the
following non-zero elements:
$$
\begin{cases}
c_{11}=1+\varepsilon^2,c_{22}=\ldots=c_{m-1,m-1}=1,\\
c_{1,m}=c_{m,1}=-\ha\varepsilon(1+\varepsilon^2+\gamma^2), c_{m,m}=\textstyle\frac{1}{4}(\varepsilon^4+2\varepsilon^2\gamma^2+2\varepsilon^2+\gamma^4-2\gamma^2+1).
\end{cases}
$$
As before, $$\det(C_\gamma)=\textstyle\frac{1}{4}(1+\varepsilon^2-\gamma^2)^2.$$
Therefore, the proof can be ended by the arguments used in subsection
\ref{case1}.
\end{proof}

\section{Proof of Theorem \ref{pr1}}

Compute the sum $[0]+\sum[\gamma]^2$. This counts almost every solution
$(y_1,\ldots ,y_m)$ of the system of equations (\ref{eq1}) and (\ref{eq2})
twice. The solutions counted only once are the special values treated in
the previous section and the doubly special values $\beta=\gamma=0$ mentioned
in Section \ref{prA}. Therefore,
$$\mu =\frac{1}{2}\left([0]+\sum [\gamma^2]+r_q+\ell_q+\sigma_0 \right).$$

\section*{Acknowledgments}

In the first part of this research, the authors were partially supported by
the Slovenian-Hungarian OTKA-ARRS Grant NN 114614 (in Hungary) and N1-0032
(in Slovenia). That project ended in May, 2019. Afterwards, the third author
was supported by Project no. ED$\underbar{}$ 18-1-2019-0030
(Application-specific highly reliable IT solutions), which has been
implemented with the support provided from the National Research, Development
and Innovation Fund of Hungary, financed under the Thematic Excellence
Programme funding scheme. This work was also partially supported by GNSAGA of INdAM. 

We are grateful to the anonymous referees for good questions,
useful suggestions and for pointing out several inaccuracies in the earlier
version of this paper.

\end{document}